\documentclass[11pt,reqno]{amsart}

\usepackage[utf8]{inputenc}
\usepackage[T1]{fontenc}
\usepackage{amsmath,amssymb,amsthm,mathtools}
\usepackage[protrusion=true,expansion=false]{microtype}
\usepackage{enumitem}
\usepackage{mathrsfs}
\usepackage[colorlinks=true,linkcolor=blue,citecolor=blue,urlcolor=blue]{hyperref}

\usepackage[margin=1in]{geometry}

\theoremstyle{plain}
\newtheorem{theorem}{Theorem}[section]
\newtheorem{proposition}[theorem]{Proposition}
\newtheorem{corollary}[theorem]{Corollary}
\newtheorem{lemma}[theorem]{Lemma}

\theoremstyle{definition}
\newtheorem{definition}[theorem]{Definition}
\newtheorem{example}[theorem]{Example}

\theoremstyle{remark}
\newtheorem{remark}[theorem]{Remark}

\newcommand{\N}{\mathbb{N}}
\newcommand{\Nz}{\mathbb{N}_{0}}
\newcommand{\Z}{\mathbb{Z}}
\newcommand{\R}{\mathbb{R}}
\newcommand{\Q}{\mathbb{Q}}

\newcommand{\one}{\mathbf{1}}
\newcommand{\udens}{\overline{d}}            
\newcommand{\ldens}{\underline{d}}           
\newcommand{\dens}{d}                        
\DeclareMathOperator{\md}{md}                
\newcommand{\umd}{\overline{\md}}            
\newcommand{\lmd}{\underline{\md}}           
\newcommand{\vp}[1][p]{v_{#1}}   

\begin{document}

\title[Additive and multiplicative densities: prime valuations and symbolic models]{Additive and multiplicative densities, prime valuations and symbolic models}

\author{N. Mora Cuellar}
\address{Department of Mathematics, University of British Columbia}
\email{natalia.mora@math.ubc.ca}

\author{I. Rojas Aravena}
\address{Department of Mathematics, University of British Columbia}
\email{i.andres@math.ubc.ca}

\author{A. Yavicoli}
\address{Department of Mathematics, University of British Columbia}
\email{yavicoli@math.ubc.ca}

\begin{abstract}
We study additive and multiplicative densities of subsets of $\N$ along
prescribed F{\o}lner sequences. We prove that additive upper density one
implies multiplicative density one along a suitable multiplicative
F{\o}lner sequence. We also prove independence in the following sense:
given an additive F{\o}lner sequence
$(G_n)_n$, a multiplicative F{\o}lner sequence $(F_n)_n$, and any
$(\alpha,\beta)\in[0,1]^2$, we construct a single set $A\subseteq\N$ such that
$\dens_{(G_n)_n}(A)=\alpha$ and $\md_{(F_n)_n}(A)=\beta$. Prime-valuation coordinates yield exact density formulas and random
models for one local condition and, under summability assumptions,
countably many; in the finite-coordinate case they also give exact higher-order
correlations. Finally, we realize these multiplicative correlations as
correlations in symbolic dynamical systems and obtain criteria for ergodicity
and mixing.
\end{abstract}

\maketitle

\section{Introduction}

A sequence $(E_n)_n$ of finite non-empty subsets of $\N$ is called an
\emph{additive F{\o}lner sequence} if, for every $m\in\N$,
\[
  \lim_{n\to\infty}\frac{|(E_n+m)\cap E_n|}{|E_n|}=1,
\]
and it is called a \emph{multiplicative F{\o}lner sequence} if, for every
$m\in\N$,
\[
  \lim_{n\to\infty}\frac{|mE_n\cap E_n|}{|E_n|}=1.
\]
Here $E_n+m:=\{x+m:x\in E_n\}$ and $mE_n:=\{mx:x\in E_n\}$.
Natural density measures a set $A\subseteq\N$ by additive averaging:
\[
  \dens(A)=\lim_{N\to\infty}\frac{|A\cap[1,N]|}{N},
\]
when the limit exists. More generally, an additive F{\o}lner sequence
$(G_n)_n$ in $(\N,+)$ defines the density
\[
  \dens_{(G_n)_n}(A)=\lim_{n\to\infty}\frac{|A\cap G_n|}{|G_n|}.
\]
From the multiplicative point of view, a multiplicative F{\o}lner sequence
$(F_n)_n$ in $(\N,\cdot)$ gives
\[
  \md_{(F_n)_n}(A)=\lim_{n\to\infty}\frac{|A\cap F_n|}{|F_n|}.
\]
The interaction between additive and multiplicative notions of largeness has a
long history in Ramsey theory, combinatorial number theory, and ergodic theory;
see, among others,
\cite{Bergelson,BergelsonHindman,BergelsonGlasscock,BergelsonMoreira}.
Multiplicative correlation sequences are studied in
\cite{DonosoLeMoreiraSun}, while normality and genericity along F{\o}lner
sequences in amenable semigroups are developed in
\cite{BergelsonDownarowiczMisiurewicz}.

The objective of the first part of the paper is to establish several relations
between additive and multiplicative density. Proposition~\ref{prop:additive-implies-mult} illustrates this by showing that if
\(\udens(A)=1\), then there is a multiplicative F{\o}lner sequence \((H_n)_n\)
for which \(\md_{(H_n)_n}(A)=1\). Moreira gave an argument in the informal note
\cite{Moreira2}; a published route is discussed in
\cite{BergelsonDownarowiczMisiurewicz}, via
\cite{BergelsonMoreiraAffine}. We then show that these two concepts are independent of each other: 
Theorem~\ref{thm:joint-range-general} states that for \emph{every} additive
F{\o}lner sequence $(G_n)_n$, every multiplicative
F{\o}lner sequence $(F_n)_n$, and every $(\alpha,\beta)\in[0,1]^2$, there is a
single set $A\subseteq\N$ such that
\[
  \dens_{(G_n)_n}(A)=\alpha
  \qquad\text{and}\qquad
  \md_{(F_n)_n}(A)=\beta.
\]

The remainder of the paper develops valuation
coordinates for multiplicative averaging and uses them to compute correlations.
For a prime $p$, write $v_p(n)$ for the exponent of
$p$ in $n$. Unique factorization identifies \((\N,\cdot)\) as the direct sum of copies of
\((\Nz,+)\) through \(n\mapsto(v_p(n))_p\). If
\[
  A_E:=\{n\in\N:v_p(n)\in E\},\qquad E\subseteq\Nz,
\]
then along a prime-box F{\o}lner sequence the multiplicative density of $A_E$
is the ordinary asymptotic density of $E$, whereas
\[
  \dens(A_E)=\sum_{e\in E}\frac{p-1}{p^{e+1}}.
\]
The one-prime and finite-prime identities are elementary consequences of
unique factorization; they provide the common mechanism for the infinite-prime,
probabilistic, and correlation results. Section~\ref{sec:symbolic} applies the standard symbolic correspondence construction, as in
\cite[Chapter~7]{EinsiedlerWard}, and uses the valuation identities to obtain
exact rather than merely one-sided correlation formulas. This yields concrete criteria for ergodicity, weak mixing, and mixing, as well
as higher-order correlation formulas under mixing of all orders.


The paper is organized as follows. Section~\ref{sec:prelim} fixes notation,
proves equidistribution of residue classes along additive F{\o}lner sequences,
and realizes arbitrary density values separately along prescribed additive and
multiplicative F{\o}lner sequences. Section~\ref{sec:relations} proves the
relation between both densities, followed by the general joint-range
theorem. Section~\ref{sec:valuation}
develops the valuation formulas and random models. Section~\ref{sec:symbolic}
establishes the symbolic correspondence and its ergodic consequences.

\section{Preliminaries}\label{sec:prelim}

In this section we introduce the two general notions of largeness used in this
paper. The results are usually established in a more general setting, but for
simplicity we restrict ourselves to the two semigroups $(\N,+)$ and
$(\N,\cdot)$. For background on multiplicative largeness, see
\cite{Bergelson}; for a systematic treatment of the interaction between
additive and multiplicative notions of largeness in a general semigroup
framework, see \cite{BergelsonGlasscock}.

Let $S\subset\N$ and $m\in\N$. We define
\begin{equation}\label{eq:shift-quotient}
   S-m:=\{n\in\N:m+n\in S\}\qquad\text{and}\qquad S/m:=\{n\in\N:mn\in S\}.
\end{equation}
\begin{definition}\label{def:folner}
An \emph{additive F{\o}lner sequence} $(F_n)_n$ is a sequence of finite non-empty
subsets of $\N$ such that for every $m\in\N$
\[
  \lim_{n\to\infty}\frac{|(F_n+m)\cap F_n|}{|F_n|}
  =\lim_{n\to\infty}\frac{|F_n\cap(F_n-m)|}{|F_n|}=1.
\]
Respectively, a \emph{multiplicative F{\o}lner sequence} is a sequence $(F_n)_n$
of finite non-empty subsets of $\N$ such that for every $m\in\N$
\[
  \lim_{n\to\infty}\frac{|mF_n\cap F_n|}{|F_n|}
  =\lim_{n\to\infty}\frac{|F_n\cap(F_n/m)|}{|F_n|}=1.
\]
\end{definition}

For every finite $F\subseteq\N$ and every $m\in\N$, the restricted
map
\[
  (F+m)\cap F\longrightarrow F\cap(F-m),
  \qquad x\longmapsto x-m,
\]
is a bijection, with inverse $y\mapsto y+m$. Likewise, the restricted map
\[
  mF\cap F\longrightarrow F\cap(F/m),
  \qquad x\longmapsto x/m,
\]
is a bijection, with inverse $y\mapsto my$. Notice that every element of
$mF\cap F$ is divisible by $m$, so the quotient $x/m$ belongs to $\N$.
Consequently,
\[
  |(F+m)\cap F|=|F\cap(F-m)|
  \qquad\text{and}\qquad
  |mF\cap F|=|F\cap(F/m)|.
\]
Thus the paired conditions in each definition are exactly equivalent, not only
asymptotically. We retain both formulations for notational symmetry. Since a
translate or dilate of a finite set has the same cardinality as the original
set, the overlap formulation is also equivalent to
\[
  \frac{|(F_n+m)\mathbin{\triangle}F_n|}{|F_n|}\longrightarrow0
  \quad\text{or}\quad
  \frac{|mF_n\mathbin{\triangle}F_n|}{|F_n|}\longrightarrow0,
\]
respectively.

We will repeatedly use two elementary consequences of the F{\o}lner property.

\begin{proposition}\label{prop:residue-classes}
Every additive or multiplicative F{\o}lner sequence $(F_n)_n$ satisfies
$|F_n|\to\infty$. Moreover, if $(G_n)_n$ is additive, then for every
$q\in\N$ and every residue $r\in\{0,\dots,q-1\}$,
\[
  \lim_{n\to\infty}
  \frac{|\{m\in G_n:m\equiv r\pmod q\}|}{|G_n|}=\frac1q.
\]
In particular, finite subsets of $\N$ have density zero along every additive
F{\o}lner sequence, while $q\N$ has density $1/q$ along every such sequence.
\end{proposition}

\begin{proof}
If $F\subseteq\N$ is finite and nonempty, then its least element does not
belong to $F+1$, and it does not belong to $2F$. Consequently,
\[
  |(F+1)\cap F|\le |F|-1,
  \qquad |2F\cap F|\le |F|-1.
\]
Applying the relevant estimate to a F{\o}lner sequence shows that
$1/|F_n|\to0$.

Now let $(G_n)_n$ be additive and, for $0\le r<q$, put
\[
  c_{n,r}:=\frac{|\{m\in G_n:m\equiv r\pmod q\}|}{|G_n|}.
\]
Translation by $1$ carries the part of $G_n$ in residue class $r$ bijectively
onto the part of $G_n+1$ in residue class $r+1$ (indices being read modulo
$q$). Therefore
\[
  |c_{n,r+1}-c_{n,r}|
  \le \frac{|(G_n+1)\mathbin{\triangle}G_n|}{|G_n|}\longrightarrow0.
\]
Since $\sum_{r=0}^{q-1}c_{n,r}=1$, all $q$ proportions differ from their
average $1/q$ by a quantity tending to zero. This proves the residue-class
statement. The final assertions follow immediately.
\end{proof}

Before turning to examples, we record a basic invariance property of
F{\o}lner sequences.

\begin{proposition}\label{prop:translate}
Let $(F_n)_n$ be an additive F{\o}lner sequence and let $(a_n)_n$ be any sequence
in $\N$. Then $(a_n+F_n)_n$ is also an additive F{\o}lner sequence. Similarly, if
$(F_n)_n$ is a multiplicative F{\o}lner sequence and $(a_n)_n$ is any sequence in
$\N$, then $(a_nF_n)_n$ is also a multiplicative F{\o}lner sequence.
\end{proposition}

\begin{proof}
We prove the multiplicative case; the additive case is identical. Fix $m\in\N$.
Since multiplication by $a_n$ is injective,
\[
  |ma_nF_n\cap a_nF_n|=|mF_n\cap F_n|,
\]
and the result follows.
\end{proof}

\begin{proposition}\label{prop:primebox-folner}
Let $(p_i)_{i\ge1}$ be an enumeration of the primes. For every
$n\in\N$ and $1\le i\le n$, let $k_{i,n}\in\Nz$, and assume that
$k_{i,n}\to\infty$ as $n\to\infty$ for every fixed $i$. Let
$(a_n)_n\subseteq\N$, and define
\[
F_n:=a_n\{p_1^{e_1}\cdots p_n^{e_n}:0\le e_i\le k_{i,n}
\text{ for }1\le i\le n\}.
\]
The sequence $(F_n)_n$ is a F{\o}lner sequence for $(\N,\cdot)$.
\end{proposition}

\begin{proof}
We first treat the case $a_n=1$. Fix $a\in\N$. Choose $r\in\N$ such that
every prime divisor of $a$ belongs to $\{p_1,\dots,p_r\}$, and write
\[
  a=\prod_{i=1}^{r}p_i^{b_i},
\]
where $b_i\ge0$ for all $i$. For $n$ sufficiently large, we have $n\ge r$ and $k_{i,n}\ge b_i$ for
$1\le i\le r$. An element $m=p_1^{e_1}\cdots p_n^{e_n}\in F_n$ then belongs
to $F_n/a$ if and only if
\[
  e_i+b_i\le k_{i,n}\qquad\forall\,1\le i\le r.
\]
Therefore
\[
  |(F_n/a)\cap F_n|
  =\Biggl(\prod_{i=1}^{r}(k_{i,n}+1-b_i)\Biggr)
   \Biggl(\prod_{j=r+1}^{n}(k_{j,n}+1)\Biggr),
\]
while
\[
  |F_n|=\prod_{j=1}^{n}(k_{j,n}+1).
\]
Hence
\[
  \frac{|(F_n/a)\cap F_n|}{|F_n|}
  =\prod_{i=1}^{r}\frac{k_{i,n}+1-b_i}{k_{i,n}+1}\longrightarrow 1,
\]
which proves the claim for the case $a_n=1$. The other case follows from
Proposition~\ref{prop:translate}.
\end{proof}

\begin{remark}\label{rem:primebox}
In the particular case where
\[
  F_n=\{p_1^{e_1}\cdots p_n^{e_n}:0\le e_1,\dots,e_n\le n\},
\]
we call the F{\o}lner sequence $(F_n)_n$ the \emph{prime-box F{\o}lner sequence}
associated to the enumeration of the primes $(p_j)_j$. This example has been
mentioned before in~\cite{Bergelson}.
\end{remark}

\begin{definition}\label{def:density}
Let $A\subseteq\N$. If $(G_n)_n$ is an additive F{\o}lner sequence, define
\[
  \udens_{(G_n)_n}(A):=\limsup_{n\to\infty}\frac{|A\cap G_n|}{|G_n|}
  \qquad\text{and}\qquad
  \ldens_{(G_n)_n}(A):=\liminf_{n\to\infty}\frac{|A\cap G_n|}{|G_n|}.
\]
When these quantities agree, we write
\[
  \dens_{(G_n)_n}(A):=\lim_{n\to\infty}\frac{|A\cap G_n|}{|G_n|}.
\]
If $(F_n)_n$ is a multiplicative F{\o}lner sequence, define analogously
\[
  \umd_{(F_n)_n}(A):=\limsup_{n\to\infty}\frac{|A\cap F_n|}{|F_n|}
  \qquad\text{and}\qquad
  \lmd_{(F_n)_n}(A):=\liminf_{n\to\infty}\frac{|A\cap F_n|}{|F_n|},
\]
and, when the limit exists,
\[
  \md_{(F_n)_n}(A):=\lim_{n\to\infty}\frac{|A\cap F_n|}{|F_n|}.
\]
The limiting formulas are identical; the notation records whether the averaging
sequence is additive or multiplicative. When $G_n=[1,n]$, we omit the subscript
and write $\udens(A)$, $\ldens(A)$, and $\dens(A)$ for the usual upper natural
density, lower natural density, and natural density.
\end{definition}

We end this section by recording that arbitrary additive- and
multiplicative-density values can be realized separately. We use the following
form of Weyl's criterion. Notice that the criterion itself only concerns a
sequence of finite averaging sets; the F{\o}lner property will enter later when
we prove that the relevant exponential averages are asymptotically invariant.

\begin{definition}
Let $(H_n)_n$ be finite nonempty subsets of $\N$ with $|H_n|\to\infty$. A
sequence $(x_m)_{m\in\N}\subset[0,1)$ is said to be \emph{equidistributed along
$(H_n)_n$} if, for every half-open interval $I=[a,b)\subseteq[0,1)$,
\[
  \lim_{n\to\infty}
  \frac{\#\{m\in H_n:x_m\in I\}}{|H_n|}=b-a.
\]
\end{definition}

\begin{lemma}[Weyl's criterion]\label{lem:weyl}
Let $(H_n)_n$ be finite nonempty subsets of $\N$ with $|H_n|\to\infty$. A
sequence $(x_m)_{m\in\N}\subset[0,1)$ is equidistributed along $(H_n)_n$ if and
only if
\[
  \lim_{n\to\infty}\frac1{|H_n|}
  \sum_{m\in H_n}e^{2\pi i kx_m}=0
\]
for every nonzero integer $k$.
\end{lemma}

\begin{proof}
Let
\[
  \nu_n:=\frac1{|H_n|}\sum_{m\in H_n}\delta_{x_m}
\]
be the empirical probability measure on the circle $\R/\Z$. If the points are
equidistributed, then $\nu_n$ converges weak-$*$ to Lebesgue measure, so every
nonzero Fourier coefficient tends to zero. Conversely, if all nonzero Fourier
coefficients tend to zero, then the integrals of every trigonometric polynomial
converge to its Lebesgue integral. Trigonometric polynomials are uniformly
dense in the continuous functions on the circle, so $\nu_n$ converges weak-$*$
to Lebesgue measure. Applying this to intervals, whose boundaries have
Lebesgue measure zero, gives equidistribution.
\end{proof}

\begin{theorem}\label{theorem:everydensityisrealized}
Let $(G_n)_n$ be an additive F{\o}lner sequence and let $(F_n)_n$ be a
multiplicative F{\o}lner sequence. For every $\alpha,\beta\in[0,1]$, there exist
sets $R,M\subseteq\N$ such that
\[
  \dens_{(G_n)_n}(R)=\alpha
  \qquad\text{and}\qquad
  \md_{(F_n)_n}(M)=\beta.
\]
\end{theorem}

\begin{proof}
Fix $\gamma\in\R\setminus\Q$ and define
\[
  R:=\{m\in\N:\{\gamma m\}\in[0,\alpha)\}.
\]
For each nonzero integer $k$, put
\[
  U_{n,k}:=\frac1{|G_n|}\sum_{m\in G_n}e^{2\pi i k\gamma m}.
\]
Since $(G_n)_n$ is additively F{\o}lner, the symmetric-difference
formulation following Definition~\ref{def:folner} gives
\[
  \left|
  \frac1{|G_n|}\sum_{m\in G_n+1}e^{2\pi i k\gamma m}-U_{n,k}
  \right|
  \le \frac{|(G_n+1)\mathbin{\triangle}G_n|}{|G_n|}
  \longrightarrow0.
\]
The translated sum is $e^{2\pi i k\gamma}U_{n,k}$, and hence
\[
  (1-e^{2\pi i k\gamma})U_{n,k}\longrightarrow0.
\]
Because $\gamma$ is irrational, $U_{n,k}\to0$. Weyl's criterion therefore
shows that $(\{\gamma m\})_{m\in\N}$ is equidistributed along $(G_n)_n$, and
so $\dens_{(G_n)_n}(R)=\alpha$.

For the multiplicative realization, define
\[
  M:=\bigl\{m\in\N:\{\gamma v_2(m)\}\in[0,\beta)\bigr\}.
\]
For each nonzero integer $k$, let
\[
  \chi_k(m):=e^{2\pi i k\gamma v_2(m)},
  \qquad
  V_{n,k}:=\frac1{|F_n|}\sum_{m\in F_n}\chi_k(m).
\]
Since $|\chi_k|=1$ and $(F_n)_n$ is multiplicatively F{\o}lner,
\[
  \left|
  \frac1{|F_n|}\sum_{m\in2F_n}\chi_k(m)-V_{n,k}
  \right|
  \le \frac{|2F_n\mathbin{\triangle}F_n|}{|F_n|}
  \longrightarrow0.
\]
Here the sum over $2F_n$ is indexed by $m=2r$ with $r\in F_n$, so it equals
$e^{2\pi i k\gamma}V_{n,k}$.
Consequently,
\[
  (1-e^{2\pi i k\gamma})V_{n,k}\longrightarrow0.
\]
Thus $V_{n,k}\to0$, and Weyl's criterion gives
$\md_{(F_n)_n}(M)=\beta$.
\end{proof}

\section{Joint ranges and endpoint comparisons}\label{sec:relations}
\subsection{Additive largeness implies multiplicative largeness}
The first result of this section shows that additive upper density one forces
multiplicative largeness along a suitable multiplicative F{\o}lner sequence.
More precisely, if $A\subseteq\N$ has upper natural density one, then one can
choose a multiplicative F{\o}lner sequence along which $A$ has density one. This
gives a positive answer to \cite[Question~9]{Moreira1}.

\begin{proposition}\label{prop:additive-implies-mult}
Let $A\subseteq\N$ satisfy $\udens(A)=1$. Then there exists a multiplicative
F{\o}lner sequence $(H_n)_{n\ge 1}$ such that $\md_{(H_n)_n}(A)=1$.
\end{proposition}

\begin{proof}
Let $B:=\N\setminus A$. Since $\udens(A)=1$, we have
\[
  \liminf_{N\to\infty}\frac{|B\cap[1,N]|}{N}=0.
\]
Let $(F_n)_n$ be the prime-box F{\o}lner sequence (Remark~\ref{rem:primebox}) for
some enumeration of primes $(p_j)_j$. For each $n$, write
$m_n:=\max F_n=p_1^{n}\cdots p_n^{n}$. Because the displayed lower limit is
zero, such estimates hold for arbitrarily large $N$. We may therefore choose
$N_n\ge2m_n$ so that
\[
  \frac{|B\cap[1,N_n]|}{N_n}\le\frac{1}{n\,m_n}.
\]
Define $T_n=\left\lfloor N_n/m_n\right\rfloor$. Then $T_n\ge 2$, and the
inequality $N_n/m_n\ge2$ also gives
\[
  T_n\ge \frac{N_n}{m_n}-1\ge\frac{N_n}{2m_n}.
\]
Moreover, if
$1\le j\le T_n$ and $r\in F_n$, then
\begin{equation}\label{eq:bound-jr}
  jr\le T_nm_n\le N_n.
\end{equation}
We now average over $j\in\{1,\dots,T_n\}$. By double counting,
\[
  \frac{1}{T_n}\sum_{j=1}^{T_n}\frac{|jF_n\cap B|}{|F_n|}
  =\frac{1}{T_n|F_n|}\sum_{j=1}^{T_n}\sum_{r\in F_n}\one_B(jr)
  =\frac{1}{T_n|F_n|}\sum_{r\in F_n}\#\{j\in\{1,\dots,T_n\}:jr\in B\}.
\]
For fixed $r$, the map $j\mapsto jr$ is injective. Together
with~\eqref{eq:bound-jr}, this gives
\[
  \#\{j\in\{1,\dots,T_n\}:jr\in B\}\le|B\cap[1,N_n]|.
\]
Therefore
\[
  \frac{1}{T_n}\sum_{j=1}^{T_n}\frac{|jF_n\cap B|}{|F_n|}\le\frac{|B\cap[1,N_n]|}{T_n}.
\]
Using $T_n\ge N_n/(2m_n)$, it follows that
\[
  \frac{|B\cap[1,N_n]|}{T_n}\le\frac{|B\cap[1,N_n]|}{N_n}\cdot 2m_n\le\frac{2}{n}.
\]
Hence
\[
  \frac{1}{T_n}\sum_{j=1}^{T_n}\frac{|jF_n\cap B|}{|F_n|}\le\frac{2}{n}.
\]
In particular, there exists $t_n\in\{1,\dots,T_n\}$ such that
\[
  \frac{|t_nF_n\cap B|}{|F_n|}\le\frac{2}{n},
\]
equivalently
\[
  \frac{|t_nF_n\cap A|}{|F_n|}\ge 1-\frac{2}{n}.
\]
Defining $H_n=t_nF_n$, we obtain a multiplicative F{\o}lner sequence by Proposition~\ref{prop:translate}, and
\[
  \md_{(H_n)_n}(A)=\lim_{n\to\infty}\frac{|H_n\cap A|}{|H_n|}=1. \qedhere
\]
\end{proof}

Proposition~\ref{prop:additive-implies-mult} gives a stronger conclusion than
the one asked for in Moreira's question. Consider the multiplicative upper
Banach density defined as
\[
  \md^{*}(E)=\sup\bigl\{\umd_{(F_n)_n}(E):(F_n)_n\text{ is a F{\o}lner sequence in }(\N,\cdot)\bigr\}.
\]
We have just proved that if $\udens(A)=1$, then $\md^{*}(A)=1$.

\begin{remark}\label{rem:moreira-alt}
Moreira gave a different argument in~\cite{Moreira2}. Related published results
appear in \cite{BergelsonMoreiraAffine}; see also
\cite{BergelsonDownarowiczMisiurewicz}. The proof above is included
because it is short, self-contained, and fits naturally with the
F{\o}lner-sequence viewpoint used throughout this paper.
\end{remark}

The hypothesis $\udens(A)=1$ is sharp: for every $\varepsilon>0$,
Moreira constructed a set $A$ with $\dens(A)\ge1-\varepsilon$ and
$\md^*(A)=0$; see \cite[Proposition~7]{Moreira1}.

\subsection{Prescribed joint ranges}
We now prove that additive and multiplicative densities can be prescribed
independently, even when both F{\o}lner sequences are fixed in advance. 

\begin{theorem}\label{thm:joint-range-general}
Let $(G_n)_n$ be an additive F{\o}lner sequence and let $(F_n)_n$ be a
multiplicative F{\o}lner sequence. Then
\[
  \bigl\{(\dens_{(G_n)_n}(A),\md_{(F_n)_n}(A)):A\subseteq\N\bigr\}=[0,1]^2,
\]
where both densities are required to exist. Equivalently, every
$(\alpha,\beta)\in[0,1]^2$ is realized by a single set $A\subseteq\N$.
\end{theorem}

Fix for the moment a multiplicative F{\o}lner sequence $(F_n)_n$, and write
$\md$ for $\md_{(F_n)_n}$.

\begin{lemma}\label{lem:carrier}
For every
$\beta\in[0,1]$ there exists a set $K\subseteq\N$ such that
\[
  \md(K)=\beta
  \qquad\text{and}\qquad
  \dens_{(G_n)_n}(K)=0
\]
for every additive F{\o}lner sequence $(G_n)_n$.
\end{lemma}

\begin{proof}
Let $C\subseteq \N$ be a subset such that  
$$\md(C) = 1 \quad \text{ and } \quad \dens_{(G_n)_n}(C) = 0,$$
for every additive Følner sequence $(G_n)_n$. Its existence is guaranteed by \cite[Theorem~7.1]{BergelsonDownarowiczMisiurewicz}. Let $\beta\in [0,1]$. By Theorem~\ref{theorem:everydensityisrealized}, we can consider a set $D \subset \N$ such that $\md(D) = \beta.$ Put $K:=D\cap C$. Since
$\md(C)=1$,
\[
  0\le\frac{|D\cap F_n|}{|F_n|}-\frac{|K\cap F_n|}{|F_n|}
  \le\frac{|F_n\setminus C|}{|F_n|}\longrightarrow0,
\]
so $\md(K)=\beta$. Moreover, $K\subseteq C$, and therefore $K$ has density
zero along every additive F{\o}lner sequence.
\end{proof}

\begin{proof}[Proof of Theorem~\ref{thm:joint-range-general}]
Fix $(\alpha,\beta)\in[0,1]^2$. Let $C$ be a set such that $\md(C) = 1$ and $\dens_{(G_n)_n}(C) = 0$.  By Theorem~\ref{theorem:everydensityisrealized}, choose $R,M\subseteq\N$ with
\[
  \dens_{(G_n)_n}(R)=\alpha,
  \qquad \md_{(F_n)_n}(M)=\beta.
\]
Separate their additive and multiplicative contributions by setting
\[
  B:=R\setminus C,
  \qquad K:=M\cap C,
  \qquad A:=B\cup K.
\]
The sets $B$ and $K$ are disjoint. Since $R\mathbin{\triangle}B\subseteq C$ and
$K\subseteq C$, the normalized counts of $R$ and $B$ along $(G_n)_n$ differ by
at most $|C\cap G_n|/|G_n|$, while $K$ contributes density zero. Consequently,
\[
  \dens_{(G_n)_n}(A)=\dens_{(G_n)_n}(B)=\alpha.
\]
On the multiplicative side, $B\subseteq\N\setminus C$, so $\md(B)=0$; also
$M\mathbin{\triangle}K=M\setminus C\subseteq\N\setminus C$, so $\md(K)=\beta$.
Thus
\[
  \md_{(F_n)_n}(A)=\beta,
\]
as required.
\end{proof}

\section{Valuation theorems}\label{sec:valuation}

\subsection{One-prime valuation}

Prime valuations linearize multiplication through
\[
  v_p(am)=v_p(a)+v_p(m).
\]
We begin with one valuation coordinate. Fix a prime $p$, let
$(p_n)_{n\ge1}$ be an enumeration of the primes with $p_1=p$, and let $(F_n)_n$
be the associated prime-box F{\o}lner sequence from
Remark~\ref{rem:primebox}.

For a subset $E\subseteq\Nz$, define
\[
  A_E:=\{m\in\N:\vp(m)\in E\}.
\]
We write
\[
  d_0(E):=\lim_{n\to\infty}\frac{|E\cap\{0,\dots,n\}|}{n+1},
\]
whenever this limit exists. Thus $d_0$ is the ordinary natural density on
$\Nz$; the subscript records that the averaging starts at $0$. For
$E\subseteq\Nz$ and $t\in\Z$, we use the shift notation
\[
  E-t:=\{e\in\Nz:e+t\in E\}.
\]
When an intersection is indexed by the empty set, it is understood to be
$\Nz$, whose $d_0$-density is $1$.

\begin{theorem}\label{thm:one-prime}
Let $E\subseteq\Nz$.
\begin{enumerate}[label=\textup{(\roman*)}]
\item For every $n\ge 1$,
\[
  \frac{|A_E\cap F_n|}{|F_n|}=\frac{|E\cap\{0,\dots,n\}|}{n+1}.
\]
In particular, the multiplicative density $\md_{(F_n)_n}(A_E)$ exists if and only
if $d_0(E)$ exists. In that case $\md_{(F_n)_n}(A_E)=d_0(E)$.
\item The natural density of $A_E$ exists and is given by
\[
  \dens(A_E)=\sum_{e\in E}\frac{p-1}{p^{e+1}}.
\]
\end{enumerate}
\end{theorem}

\begin{remark}\label{rem:convention}
Averaging instead over $\{1,\dots,n\}$ gives the same limit, whenever either
limit exists.
\end{remark}

\begin{proof}
(i) Every element $m\in F_n$ has a unique representation
$m=p_1^{e_1}\cdots p_n^{e_n}$, $0\le e_1,\dots,e_n\le n$. Since $p_1=p$, we have
$\vp(m)=e_1$. Therefore
\[
  m\in A_E\cap F_n\iff e_1\in E\cap\{0,\dots,n\}.
\]
For each admissible value of $e_1$, there are exactly $(n+1)^{n-1}$ choices for
$e_2,\dots,e_n$. Hence $|A_E\cap F_n|=|E\cap\{0,\dots,n\}|(n+1)^{n-1}$, while
$|F_n|=(n+1)^{n}$. Thus
\[
  \frac{|A_E\cap F_n|}{|F_n|}=\frac{|E\cap\{0,\dots,n\}|}{n+1}.
\]

(ii) For each $e\in\Nz$, define
\[
  A_e:=A_{\{e\}}=\{m\in\N:\vp(m)=e\}=p^{e}\N\setminus p^{e+1}\N.
\]
Then $\dens(A_e)=\dfrac{1}{p^{e}}-\dfrac{1}{p^{e+1}}=\dfrac{p-1}{p^{e+1}}$. The
sets $A_e$ are pairwise disjoint and $A_E=\bigsqcup_{e\in E}A_e$. For $M\ge0$,
let $E_M:=E\cap\{0,\dots,M\}$. Then
\[
  \bigsqcup_{e\in E_M}A_e\subseteq A_E
  \subseteq\Bigl(\bigsqcup_{e\in E_M}A_e\Bigr)\sqcup\Bigl(\bigsqcup_{e>M}A_e\Bigr).
\]
Moreover $\bigsqcup_{e>M}A_e=p^{M+1}\N$, so this tail has natural density
$p^{-(M+1)}$. Hence
\[
  \sum_{e\in E_M}\frac{p-1}{p^{e+1}}\le\ldens(A_E)\le\udens(A_E)
  \le\sum_{e\in E_M}\frac{p-1}{p^{e+1}}+\frac{1}{p^{M+1}}.
\]
Letting $M\to\infty$ gives
\[
\dens(A_E)=\sum_{e\in E}\frac{p-1}{p^{e+1}}.
\]
\end{proof}

With this result, we are now able to construct an explicit example illustrating
the independence phenomenon established in Section~\ref{sec:relations}.

\begin{example}\label{ex:v2-odd}
Let $A:=\{n\in\N:v_2(n)\text{ is odd}\}$. Then
\[
  \dens(A)=\frac{1}{3}\qquad\text{and}\qquad\md_{(F_n)_n}(A)=\frac{1}{2}.
\]
Indeed, take $p=2$ and $E=\{1,3,5,\dots\}$, the set of positive odd integers.
Then $d_0(E)=1/2$, so Theorem~\ref{thm:one-prime}(i) gives
$\md_{(F_n)_n}(A)=1/2$. Also
\[
  \dens(A)=\sum_{j=0}^{\infty}\frac{1}{2^{2j+2}}=\frac{1}{3}.
\]
\end{example}

We end this subsection with a simple probabilistic model showing that the
prescribed multiplicative density occurs generically.

\begin{proposition}\label{prop:random}
Fix a prime $p$, and let $(X_e)_{e\ge 0}$ be independent Bernoulli random
variables with
\[
  \mathbb{P}(X_e=1)=\theta,\qquad\mathbb{P}(X_e=0)=1-\theta,
\]
where $\theta\in[0,1]$. Define the random exponent set
$E_\omega:=\{e\in\Nz:X_e(\omega)=1\}$ and the corresponding random valuation set
$A_\omega:=\{n\in\N:\vp(n)\in E_\omega\}$. Then:
\begin{enumerate}[label=\textup{(\roman*)}]
\item with probability one, the asymptotic density of $E_\omega$ exists and
  equals $\theta$;
\item with probability one, the multiplicative density
  $\md_{(F_n)_n}(A_\omega)$ exists and equals $\theta$, where $(F_n)_n$ is the
  prime-box multiplicative F{\o}lner sequence associated to $p$;
\item for every $\omega$, the natural density of $A_\omega$ exists and satisfies
\[
  \dens(A_\omega)=\sum_{e=0}^{\infty}X_e(\omega)\,\frac{p-1}{p^{e+1}}.
\]
\end{enumerate}
Moreover, $\mathbb{E}[\dens(A_\omega)]=\theta$.
\end{proposition}

\begin{proof}
By the strong law of large numbers, $\sum_{e=0}^{n}X_e(\omega)/(n + 1)\to\theta$
for almost every $\omega$. Since $\sum_{e=0}^{n}X_e(\omega)=|E_\omega\cap\{0,\dots,n\}|$,
this proves (i), and (ii) then follows from
Theorem~\ref{thm:one-prime}(i). For every $\omega$, part (iii) follows directly
from Theorem~\ref{thm:one-prime}(ii), applied to the deterministic set
$E_\omega$. Finally, since
$\sum_{e=0}^{\infty}\frac{p-1}{p^{e+1}}=1$, Tonelli's theorem gives
\[
  \mathbb{E}[\dens(A_\omega)]=\sum_{e=0}^{\infty}\mathbb{E}[X_e]\,\frac{p-1}{p^{e+1}}
  =\theta\sum_{e=0}^{\infty}\frac{p-1}{p^{e+1}}=\theta. \qedhere
\]
\end{proof}

\subsection{Finite-prime valuation}\label{subsec:finite-prime}

Let $k\in\N$, let $(p_n)_{n\ge1}$ be an enumeration of the primes, and use the
prime-box F{\o}lner sequence from Remark~\ref{rem:primebox}. In the box, the
first $k$ exponent coordinates vary independently and uniformly. For natural
density, the corresponding factorization comes instead from the independence
of divisibility conditions at distinct primes.

\begin{theorem}\label{thm:finite-prime}
For each $i\in\{1,\dots,k\}$,
let $E_i\subseteq\Nz$, and define
\[
  A_i:=\{m\in\N:\vp[p_i](m)\in E_i\},\qquad A:=\bigcap_{i=1}^{k}A_i.
\]
Then:
\begin{enumerate}[label=\textup{(\roman*)}]
\item For every $n\ge k$,
\[
  \frac{|A\cap F_n|}{|F_n|}=\prod_{i=1}^{k}\frac{|E_i\cap\{0,\dots,n\}|}{n+1}.
\]
In particular, if the densities $d_0(E_i)$ exist for all $i=1,\dots,k$, then the
multiplicative density of $A$ with respect to $(F_n)_n$ exists and is given
by
\[
  \md_{(F_n)_n}(A)=\prod_{i=1}^{k}d_0(E_i).
\]
\item The natural density of $A$ exists and is given by
\[
  \dens(A)
  =\prod_{i=1}^{k}\Biggl(\sum_{e\in E_i}\frac{p_i-1}{p_i^{e+1}}\Biggr)
  =\prod_{i=1}^{k}\dens(A_i).
\]
\end{enumerate}
\end{theorem}

\begin{proof}
(i) Fix $n\ge k$. Every element of $F_n$ has a unique representation
$m=p_1^{e_1}\cdots p_n^{e_n}$, $0\le e_1,\dots,e_n\le n$. For every $1\le i\le k$,
we have $\vp[p_i](m)=e_i$. Thus $m\in A$ if and only if
$e_i\in E_i\cap\{0,\dots,n\}$ for $1\le i\le k$. Hence
\[
  |A\cap F_n|=\Biggl(\prod_{i=1}^{k}|E_i\cap\{0,\dots,n\}|\Biggr)(n+1)^{n-k}.
\]
Dividing by $|F_n|=(n+1)^{n}$ gives the claim.

(ii) For each $1\le i\le k$ and each $e\in\Nz$, let
$B_{i,e}:=\{m\in\N:\vp[p_i](m)=e\}$. For $\mathbf e=(e_1,\dots,e_k)\in\Nz^{k}$, set
$C_{\mathbf e}:=\bigcap_{i=1}^{k}B_{i,e_i}$. Then the sets $C_{\mathbf e}$ are
pairwise disjoint and $A=\bigsqcup_{\mathbf e\in E_1\times\cdots\times E_k}C_{\mathbf e}$.
Fix $\mathbf e=(e_1,\dots,e_k)$ and write $Q:=\prod_{i=1}^{k}p_i^{e_i}$. An integer
$m$ belongs to $C_{\mathbf e}$ if and only if $m=Qn$ and $n$ is not divisible by
any of $p_1,\dots,p_k$. By inclusion--exclusion, the set of such $n$ has natural density
$\prod_{i=1}^{k}(1-1/p_i)$. Multiplying every element of a set by $Q$ divides
its natural density by $Q$. Therefore
\[
  \dens(C_{\mathbf e})=\frac{1}{Q}\prod_{i=1}^{k}\Bigl(1-\frac{1}{p_i}\Bigr)
  =\prod_{i=1}^{k}\frac{p_i-1}{p_i^{e_i+1}}.
\]
For $M\ge0$, let
\[
  A^{(M)}:=\bigsqcup_{\mathbf e\in
  (E_1\cap[0,M])\times\cdots\times(E_k\cap[0,M])}C_{\mathbf e}.
\]
This is a finite disjoint union, and hence
\[
  \dens(A^{(M)})
  =\sum_{\mathbf e\in
  (E_1\cap[0,M])\times\cdots\times(E_k\cap[0,M])}
  \prod_{i=1}^{k}\frac{p_i-1}{p_i^{e_i+1}}.
\]
Moreover,
\[
  A\setminus A^{(M)}
  \subseteq\bigcup_{i=1}^{k}\{m:\vp[p_i](m)\ge M+1\}
  =\bigcup_{i=1}^{k}p_i^{M+1}\N,
\]
so
\[
  \udens(A\setminus A^{(M)})
  \le\sum_{i=1}^{k}p_i^{-(M+1)}\longrightarrow0.
\]
Since $A^{(M)}\subseteq A$, we have
\[
  \dens(A^{(M)})
  \le \ldens(A)\le\udens(A)
  \le \dens(A^{(M)})+\sum_{i=1}^{k}p_i^{-(M+1)}.
\]
Letting $M\to\infty$ shows that the lower and upper densities of $A$ agree and
that
\[
  \dens(A)
  =\sum_{\mathbf e\in E_1\times\cdots\times E_k}
  \prod_{i=1}^{k}\frac{p_i-1}{p_i^{e_i+1}}
  =\prod_{i=1}^{k}\Biggl(\sum_{e\in E_i}
  \frac{p_i-1}{p_i^{e+1}}\Biggr).
\]
The last equality follows from Tonelli's theorem (all terms are nonnegative),
and the final expression equals $\prod_{i=1}^{k}\dens(A_i)$ by
Theorem~\ref{thm:one-prime}(ii).
\end{proof}

The following result shows that higher-order multiplicative correlations are
completely controlled by additive intersections of shifted exponent sets. The
reason is that $m\in A/a$ means $am\in A$, and multiplication by $a$ adds the
fixed vector of valuations $(v_{p_i}(a))_i$ to the valuation vector of $m$.

\begin{proposition}\label{prop:correlations}
Let $\ell\in\N$, let $a_1,\dots,a_\ell\in\N$, and for each $j=1,\dots,\ell$ and
each $i=1,\dots,k$ set $t_{j,i}:=\vp[p_i](a_j)$ and $t_{0,i}:=0$. Assume that for
each $i$ the asymptotic density
$d_0\bigl(\bigcap_{j=0}^{\ell}(E_i-t_{j,i})\bigr)$ exists. Then the
multiplicative density of $A\cap A/a_1\cap\cdots\cap A/a_\ell$ with respect to
$(F_n)_n$ exists and equals
\[
  \md_{(F_n)_n}\bigl(A\cap A/a_1\cap\cdots\cap A/a_\ell\bigr)
  =\prod_{i=1}^{k}d_0\Bigl(\bigcap_{j=0}^{\ell}(E_i-t_{j,i})\Bigr).
\]
\end{proposition}

\begin{proof}
Let $B:=A\cap A/a_1\cap\cdots\cap A/a_\ell$. Fix $m\in\N$. Then $m\in B$ if and
only if, for every $i\in\{1,\dots,k\}$,
\[
  \vp[p_i](m)\in E_i\quad\text{and}\quad\vp[p_i](a_jm)=\vp[p_i](m)+t_{j,i}\in E_i
  \quad(1\le j\le\ell).
\]
Equivalently, $\vp[p_i](m)\in\bigcap_{j=0}^{\ell}(E_i-t_{j,i})$ for
$1\le i\le k$. Thus
\[
  B=\bigcap_{i=1}^{k}\Bigl\{m\in\N:\vp[p_i](m)\in\bigcap_{j=0}^{\ell}(E_i-t_{j,i})\Bigr\}.
\]
Now apply part (i) of Theorem~\ref{thm:finite-prime} to the new exponent sets
$E_i':=\bigcap_{j=0}^{\ell}(E_i-t_{j,i})$, which gives the stated formula.
\end{proof}

This shows that multiplicative recurrence is additive recurrence in exponent
space.

\subsection{Infinite-prime valuation}
We now impose valuation conditions at countably many primes. Throughout this
subsection, $(p_i)_{i\ge1}$ is an enumeration of the primes and $(F_n)_n$
is the associated prime-box F{\o}lner sequence. Every infinite product below is
understood as the limit of its finite partial products; because all factors lie
in $[0,1]$, this limit always exists (possibly as zero).

For $E\subseteq\Nz$ we write
\[
  \delta_n(E):=\frac{|\{0,\dots,n\}\setminus E|}{n+1}=1-\frac{|E\cap\{0,\dots,n\}|}{n+1}.
\]

\begin{theorem}\label{thm:infinite-prime}
Let $(E_i)_{i\ge1}$ be a collection of subsets of $\Nz$. Set
\[
  A_i:=\{n\in\N:\vp[p_i](n)\in E_i\},
  \qquad A:=\bigcap_{i=1}^{\infty}A_i.
\]
\begin{enumerate}[label=\textup{(\roman*)}]
    \item For every $n \in \N$
    \[
      \frac{|A\cap F_n|}{|F_n|}
      =\one_{\{\forall j>n,\ 0\in E_j\}}
       \prod_{i=1}^{n}\frac{|E_i\cap\{0,\ldots,n\}|}{n+1}.
    \]
    In particular, if the densities $d_0(E_i)$ exist for all $i\in\N$ and
    \[
      \sum_{i\ge1}\sup_{n\ge0}\delta_n(E_i)<\infty,
    \]
    then the multiplicative density of $A$ with respect to $(F_n)_n$
    exists and is given by
    \[
      \md_{(F_n)_n}(A)=\prod_{i=1}^{\infty}d_0(E_i).
    \]
    \item If
    \[
      \sum_{i=1}^{\infty}\bigl(1-\dens(A_i)\bigr)<\infty,
    \]
    then the natural density of $A$ exists and satisfies
    \[
      \dens(A)=\prod_{i=1}^{\infty}\dens(A_i).
    \]
\end{enumerate}
\end{theorem}
\begin{proof}
(i) Every $m\in F_n$ has a unique representation
$m=p_1^{e_1}\cdots p_n^{e_n}$ with $0\le e_i\le n$. Such an $m$ belongs to
$A$ if and only if $e_i\in E_i$ for $1\le i\le n$ and $0\in E_j$ for every
$j>n$. Hence
\[
  \frac{|A\cap F_n|}{|F_n|}
  =\one_{\{\forall j>n,\ 0\in E_j\}}
   \prod_{i=1}^{n}\frac{|E_i\cap\{0,\ldots,n\}|}{n+1}
  =\one_{\{\forall j>n,\ 0\in E_j\}}
   \prod_{i=1}^{n}\bigl(1-\delta_n(E_i)\bigr).
\]
Now assume that all the densities $d_0(E_i)$ exist and
\[
  \sum_{i\ge1}\sup_n\delta_n(E_i)<\infty.
\]
Only finitely many sets $E_i$ can omit $0$, because $0\notin E_i$ implies
$\delta_0(E_i)=1$. Consequently, the indicator in the preceding formula equals
$1$ for all sufficiently large $n$.

Suppose first that $d_0(E_k)=0$ for some $k\in\N$. Since
$1-\delta_n(E_k)\to0$ and every factor lies in $[0,1]$, we have
\[
  \prod_{i=1}^{n}\bigl(1-\delta_n(E_i)\bigr)\longrightarrow0
  =\prod_{i=1}^{\infty}d_0(E_i).
\]
Suppose now that $d_0(E_i)>0$ for every $i\in\N$. Choose $i_1$ so large
that $\sup_n\delta_n(E_i)\le1/2$ for every $i>i_1$ and $0\in E_i$ for every
$i\ge i_1$. For $i>i_1$, define
\[
  f_n(i):=\one_{\{i\le n\}}\log\bigl(1-\delta_n(E_i)\bigr).
\]
For every fixed $i>i_1$,
\[
  f_n(i)\longrightarrow\log d_0(E_i),
\]
and
\[
  |f_n(i)|\le2\sup_m\delta_m(E_i).
\]
The dominating sequence is summable by hypothesis. Dominated convergence on
counting measure therefore gives
\[
  \sum_{i_1<i\le n}\log\bigl(1-\delta_n(E_i)\bigr)
  \longrightarrow
  \sum_{i>i_1}\log d_0(E_i).
\]
For each of the finitely many indices $i\le i_1$, the positivity of
$d_0(E_i)$ implies that $1-\delta_n(E_i)>0$ for all sufficiently large $n$.
Thus the corresponding logarithms are eventually well defined and
\[
  \sum_{i\le i_1}\log\bigl(1-\delta_n(E_i)\bigr)
  \longrightarrow
  \sum_{i\le i_1}\log d_0(E_i).
\]
Exponentiating yields
\[
  \md_{(F_n)_n}(A)=\prod_{i=1}^{\infty}d_0(E_i).
\]

(ii) For each $m\in\N$, define the finite-prime approximation
\[
  A^{(m)}:=\{n\in\N:\vp[p_i](n)\in E_i\text{ for }1\le i\le m\}.
\]
By Theorem~\ref{thm:finite-prime}(ii),
\[
  \dens(A^{(m)})=\prod_{i=1}^{m}\dens(A_i).
\]
Let
\[
  B_i:=\N\setminus A_i=\{n\in\N:\vp[p_i](n)\notin E_i\}.
\]
Then, in fact,
\[
  A=A^{(m)}\setminus\bigcup_{i>m}B_i,
  \qquad\text{and hence}\qquad A\subseteq A^{(m)}.
\]
We estimate the upper density of the union on the left. If $0\notin E_i$, then
Theorem~\ref{thm:one-prime}(ii) gives
\[
  1-\dens(A_i)\ge\frac{p_i-1}{p_i}\ge\frac12.
\]
The summability hypothesis therefore implies that $0\in E_i$ for all
sufficiently large $i$. For such $i$,
\[
  B_i\subseteq\bigcup_{\substack{e\ge1\\e\notin E_i}}p_i^e\N.
\]
For each $N$, we may now apply the union bound inside $[1,N]$ and use
$|p_i^e\N\cap[1,N]|/N\le p_i^{-e}$. The resulting double series is finite by
the summability hypothesis. Consequently, for all sufficiently large $m$,
\begin{align*}
  \udens\Bigl(\bigcup_{i>m}B_i\Bigr)
  &\le \sum_{i>m}\sum_{\substack{e\ge1\\e\notin E_i}}\frac1{p_i^e}\\
  &=\sum_{i>m}\frac{p_i}{p_i-1}\bigl(1-\dens(A_i)\bigr)\\
  &\le2\sum_{i>m}\bigl(1-\dens(A_i)\bigr)=:\eta_m.
\end{align*}
Here $\eta_m\to0$ by assumption. The two inclusions above now give the squeeze
\[
  \dens(A^{(m)})-\eta_m
  \le\ldens(A)\le\udens(A)\le\dens(A^{(m)}).
\]
Letting $m\to\infty$ shows that the lower and upper densities agree and that
\[
  \dens(A)
  =\lim_{m\to\infty}\dens(A^{(m)})
  =\prod_{i=1}^{\infty}\dens(A_i).
\]
\end{proof}

\begin{example}\label{ex:squarefree}
Take $E_i=\{0,1\}$, for which $\dens(A_i)=1-1/p_i^2$. Since
$\sum_p p^{-2}<\infty$, Theorem~\ref{thm:infinite-prime}(ii) applies and gives
\[
  \dens(A)=\prod_{p}\Bigl(1-\frac{1}{p^{2}}\Bigr)=\frac{1}{\zeta(2)},
\]
recovering the density of the squarefree integers. Another application is with
sets defined by requiring each prime valuation to belong to a local set $E_i$
with summably small complement.
\end{example}

We now turn to a probabilistic model in which the local conditions at all primes
are chosen at random, extending Proposition~\ref{prop:random}.

\begin{remark}\label{rem:B-free}
Let
\[
  \mathscr B=\{p_i^{r_i}:i\ge1\},
\]
where the $p_i$ are distinct primes and $r_i\in\N$. The set of
$\mathscr B$-free integers can then be written as
\[
  \mathcal F_{\mathscr B}
  :=\{n\in\N:b\nmid n\text{ for every }b\in\mathscr B\}
  =\bigcap_{i\ge1}
    \{n\in\N:v_{p_i}(n)\in\{0,\dots,r_i-1\}\}.
\]
Consequently, if
\[
  \sum_{i\ge1}\frac{1}{p_i^{r_i}}<\infty,
\]
Theorem~\ref{thm:infinite-prime}(ii) recovers the classical formula
\[
  \dens(\mathcal F_{\mathscr B})
  =\prod_{i\ge1}\left(1-\frac{1}{p_i^{r_i}}\right).
\]
The choice $r_i=2$ for every $i$ gives the squarefree integers considered
in Example~\ref{ex:squarefree}.

The usual $\mathscr B$-free subshift is generated by
$\one_{\mathcal F_{\mathscr B}}$ under the additive shift, and its invariant
measures have been studied extensively; see, for example,
\cite{KulagaLemanczykWeiss}. By contrast, the valuation systems considered
here encode multiplication through translations in the exponent
coordinates. Thus the two constructions associate different dynamical
systems with the same arithmetic set.
\end{remark}

\begin{proposition}\label{prop:random-infinite}
Let $(X_{i,e})_{i\ge1,\,e\ge0}$ be independent Bernoulli random variables with
$\mathbb{P}(X_{i,e}=1)=\theta_i\in (0,1]$. Put
\[
  E_i:=\{e:X_{i,e}=1\},\qquad
  A_i:=\{n\in\N:\vp[p_i](n)\in E_i\},
\]
and define
\[
  A_\omega:=\bigcap_{i=1}^{\infty}A_i
  =\{n\in\N:\vp[p_i](n)\in E_i\text{ for every }i\ge1\}.
\]
\begin{enumerate}[label=\textup{(\roman*)}]
\item If $\sum_i(1-\theta_i)<\infty$, then almost surely
\[
  \sum_{i=1}^{\infty}\bigl(1-\dens(A_i)\bigr)<\infty,
\]
the natural density $\dens(A_\omega)$ exists and equals
$\prod_i\dens(A_i)>0$, and
\[
  \mathbb{E}\bigl[\dens(A_\omega)\bigr]=\prod_{i=1}^{\infty}\theta_i> 0.
\]
\item If $\sum_{i}(1-\theta_i)=\infty$, then almost surely $\dens(A_\omega)=0$.
\end{enumerate}
\end{proposition}

\begin{proof}
For each $i$, let
\[
  \beta_i:=\sum_{e\notin E_i}\frac{p_i-1}{p_i^{e+1}}
  =1-\dens(A_i).
\]
Then $0\le\beta_i\le1$,
\[
  \mathbb{E}[\beta_i]
  =(1-\theta_i)\sum_{e\ge0}\frac{p_i-1}{p_i^{e+1}}
  =1-\theta_i,
\]
and the random variables $\beta_i$ are independent because they depend on
disjoint families of the $X_{i,e}$.

(i) Since
\[
  \mathbb{E}\Bigl[\sum_i\beta_i\Bigr]=\sum_i(1-\theta_i)<\infty,
\]
Tonelli's theorem gives
$\mathbb E[\sum_i\beta_i]<\infty$; since the sum is nonnegative, it follows
that $\sum_i\beta_i<\infty$ almost surely. We also have
\[
  \mathbb{P}(\beta_i=1)
  =\mathbb{P}(E_i=\varnothing)
  =\lim_{M\to\infty}(1-\theta_i)^{M+1}=0.
\]
Thus, outside a single null set, $\beta_i<1$ for every $i$. On this event the
finite initial factors $1-\beta_i$ are positive, while for all sufficiently
large $i$ one has $\beta_i\le1/2$ and
\[
  |\log(1-\beta_i)|\le2\beta_i.
\]
Hence $\prod_i(1-\beta_i)>0$. Theorem~\ref{thm:infinite-prime} now yields
\[
  \dens(A_\omega)=\prod_i\gamma_i>0,
  \qquad \gamma_i:=1-\beta_i=\dens(A_i).
\]
The partial products $\prod_{i\le m}\gamma_i$ decrease to
$\dens(A_\omega)$ and lie in $[0,1]$. Moreover,
$\mathbb E[\gamma_i]=1-\mathbb E[\beta_i]=\theta_i$. Therefore, by dominated
convergence and independence,
\[
  \mathbb{E}\bigl[\dens(A_\omega)\bigr]
  =\lim_{m\to\infty}\mathbb{E}\Bigl[\prod_{i\le m}\gamma_i\Bigr]
  =\lim_{m\to\infty}\prod_{i\le m}\mathbb{E}[\gamma_i]
  =\prod_{i=1}^{\infty}\theta_i.
\]
This product is positive because every $\theta_i>0$ and, for all sufficiently
large $i$, the estimate $-\log\theta_i\le2(1-\theta_i)$ makes
$\sum_i|\log\theta_i|$ finite.

(ii) Suppose that $\sum_i(1-\theta_i)=\infty$. Let
\[
  S_n:=\sum_{i=1}^n\beta_i
  \qquad\text{and}\qquad
  a_n:=\mathbb{E}[S_n]=\sum_{i=1}^n(1-\theta_i).
\]
Then $a_n\to\infty$, and independence together with $0\le\beta_i\le1$ gives
\[
  \operatorname{Var}(S_n)
  =\sum_{i=1}^n\operatorname{Var}(\beta_i)
  \le\sum_{i=1}^n\mathbb{E}[\beta_i^2]
  \le a_n.
\]
Choose $n_k$ so that $a_{n_k}\ge k^2$. By Chebyshev's inequality,
\[
  \mathbb{P}\left(S_{n_k}<\frac{a_{n_k}}2\right)
  \le\frac{4\operatorname{Var}(S_{n_k})}{a_{n_k}^2}
  \le\frac4{k^2}.
\]
The Borel--Cantelli lemma implies that $S_{n_k}\ge a_{n_k}/2$ eventually.
Since $a_{n_k}\to\infty$ and the partial sums $S_n$ are nondecreasing, it
follows that $\sum_i\beta_i=\infty$ almost surely. Consequently, almost surely,
\[
  \prod_{i=1}^{m}(1-\beta_i)
  \le \exp\Bigl(-\sum_{i=1}^{m}\beta_i\Bigr)
  =e^{-S_m}\longrightarrow0.
\]
Define the finite approximants
\[
  A_\omega^{(m)}:=\bigcap_{i=1}^{m}A_i.
\]
Since $A_\omega\subseteq A_\omega^{(m)}$ and
\[
  \dens(A_\omega^{(m)})=\prod_{i\le m}(1-\beta_i)\longrightarrow0,
\]
we obtain
\[
  \udens(A_\omega)\le\inf_m\dens(A_\omega^{(m)})=0.
\]
Therefore $\dens(A_\omega)=0$ almost surely.
\end{proof}

Thus, under the standing assumption $\theta_i>0$ for every $i$, the random set
has positive natural density almost surely when
$\sum_i(1-\theta_i)<\infty$, and has natural density zero almost surely when
this series diverges.

\section{Symbolic correspondence}\label{sec:symbolic}

In this section we explore some consequences of the valuation dictionary from
the point of view of the Furstenberg correspondence principle. We follow the
standard symbolic construction, as presented for instance
in~\cite[Chapter 7]{EinsiedlerWard}.

Let $X_0=\{0,1\}^{\Z}$ be the full shift on two symbols, equipped with the
compact product topology induced by the discrete topology on $\{0,1\}$. We use
the convention $(\sigma x)(n)=x(n+1)$. Given $E\subseteq\Nz$, define a point
$x_E\in\{0,1\}^{\Z}$ by
\begin{equation}\label{eq:def-xE}
x_E(n):=
  \begin{cases}
    1,& n\in E,\\
    0,& n\notin E,
  \end{cases}\quad(n\ge 0),
  \qquad\text{and}\qquad x_E(n):=0\quad(n<0).
\end{equation}
We regard $\one_E$ as a function on $\Z$ by setting $\one_E(n)=0$ for $n<0$;
with this convention, $x_E(n)=\one_E(n)$ for every $n\in\Z$.

Let
\[
  X_E:=\overline{\{\sigma^{m}x_E:m\in\Z\}}\subseteq\{0,1\}^{\Z},
\]
where the closure is taken in the product topology. Then $X_E$ is invariant
under both $\sigma$ and $\sigma^{-1}$. From now on, $\sigma$ denotes the
homeomorphism obtained by restricting the shift to $X_E$.
We use the exponent-shift notation $E-t$ introduced before
Theorem~\ref{thm:one-prime}.

\subsection{One-prime case}

In this subsection, we fix a prime $p$ and use the same notation as in
Section~\ref{sec:valuation}. Throughout this subsection, every multiplicative
density is taken with respect to $\md_{(F_n)_n}$, where $(F_n)_n$ is the
prime-box F{\o}lner sequence.

The first observation is that, under this dictionary, additive recurrence in the
exponent set is translated into multiplicative recurrence for the corresponding
valuation set. However, under an additional hypothesis, we can explicitly
describe the system arising from Furstenberg's correspondence principle.
\begin{theorem}\label{thm:correspondence}
Assume that for every finite set $T\subseteq\Z$ the asymptotic density
$d_0\bigl(\bigcap_{t\in T}(E-t)\bigr)$ exists. Then there exists a
$\sigma$-invariant Borel probability measure $\mu_E$ on $X_E$ such that, with
\[
  B:=\{x\in X_E:x(0)=1\},
\]
for every finite set $T\subseteq\Z$,
\[
  \mu_E\Bigl(\bigcap_{t\in T}\sigma^{-t}B\Bigr)=d_0\Bigl(\bigcap_{t\in T}(E-t)\Bigr).
\]
Moreover, if $A_E=\{n\in\N:\vp(n)\in E\}$ and $a_1,\dots,a_\ell\in\N$, then
\[
  \md\bigl(A_E\cap A_E/a_1\cap\cdots\cap A_E/a_\ell\bigr)
  =\mu_E\bigl(B\cap\sigma^{-\vp(a_1)}B\cap\cdots\cap\sigma^{-\vp(a_\ell)}B\bigr).
\]
\end{theorem}

\begin{proof}
For $N\in\N$, define the measure
\[
\mu_N:=\frac{1}{N}\sum_{m=0}^{N-1}\delta_{\sigma^{m}x_E},
\]
where $\delta_y$ denotes the Dirac mass at $y$. We first show that the full
sequence $(\mu_N)_N$ converges. We know that any accumulation point of this sequence is $\sigma$-invariant (\cite[Theorem 4.1]{EinsiedlerWard}). Functions depending on only finitely many coordinates, namely, cylinder functions, are uniformly dense in $C(X_E)$. It is
therefore enough to prove convergence on cylinder indicators, so the Riesz representation theorem gives a unique Borel probability measure $\mu_E$ and
$\mu_N\to\mu_E$ weak-$*$.

Fix $M\in\N$ and finite disjoint sets $P,Q\subseteq\Z$ with
$P\cup Q=\{-M,-M+1,\dots,M\}$.
Consider
\[
  C(P,Q):=\{x\in X_E:x(n)=1\ \forall n\in P,\ x(n)=0\ \forall n\in Q\}.
\]
For $m\ge M$, all integers $m+n$ with $n\in P\cup Q$ are nonnegative, so
$\sigma^{m}x_E\in C(P,Q)$ is equivalent to the conjunction of the conditions
$m+n\in E$ ($n\in P$) and $m+n\notin E$ ($n\in Q$). Thus, for $m\ge M$,
\[
  \one_{C(P,Q)}(\sigma^{m}x_E)
  =\prod_{n\in P}\one_E(m+n)\prod_{n\in Q}\bigl(1-\one_E(m+n)\bigr).
\]
Expanding the second product and using that $P$ and $Q$ are disjoint (hence so
are $P$ and any $R\subseteq Q$), we obtain
\[
  \one_{C(P,Q)}(\sigma^{m}x_E)
  =\sum_{R\subseteq Q}(-1)^{|R|}\prod_{n\in P\cup R}\one_E(m+n).
\]
Hence
\[
  \frac{1}{N}\sum_{m=M}^{N-1}\one_{C(P,Q)}(\sigma^{m}x_E)
  =\sum_{R\subseteq Q}(-1)^{|R|}\frac{1}{N}\sum_{m=M}^{N-1}\prod_{n\in P\cup R}\one_E(m+n).
\]
Note that
\[
  \frac{1}{N}\sum_{m=M}^{N-1}\prod_{n\in P\cup R}\one_E(m+n)
  =\frac{\bigl|\bigcap_{n\in P\cup R}(E-n)\cap[M,N-1]\bigr|}{N}.
\]
Since $M$ is fixed, the density
$d_0\bigl(\bigcap_{n\in P\cup R}(E-n)\bigr)$ exists by assumption. With the
index change $N=L+1$, its defining average is over $[0,N-1]$ with denominator
$N$; deleting the finitely many terms $0\le m<M$ does not affect the limit.
Therefore
\[
  \lim_{N\to\infty}\frac{\bigl|\bigcap_{n\in P\cup R}(E-n)\cap[M,N-1]\bigr|}{N}
  =d_0\Bigl(\bigcap_{n\in P\cup R}(E-n)\Bigr).
\]
Thus
\[
  \lim_{N\to\infty}\frac{1}{N}\sum_{m=0}^{N-1}\one_{C(P,Q)}(\sigma^{m}x_E)
  =\sum_{R\subseteq Q}(-1)^{|R|}d_0\Bigl(\bigcap_{n\in P\cup R}(E-n)\Bigr).
\]
The cylinders $C(P,Q)$ considered above specify every coordinate in a finite
interval. Any cylinder supported on $\{-M,\dots,M\}$ but leaving some coordinates
unspecified is a finite disjoint union of such fully specified cylinders.
Therefore $\mu_N(C)$ converges for every cylinder $C$, and the density argument
at the beginning of the proof yields $\mu_N\to\mu_E$ weak-$*$. In
particular, for finite $T\subseteq\Z$,
\[
  \mu_E\Bigl(\bigcap_{t\in T}\sigma^{-t}B\Bigr)
  =\lim_{N\to\infty}\frac{1}{N}\sum_{m=0}^{N-1}\prod_{t\in T}\one_E(m+t)
  =d_0\Bigl(\bigcap_{t\in T}(E-t)\Bigr).
\]
Finally, if $A_E=\{n:\vp(n)\in E\}$ and $a_1,\dots,a_\ell\in\N$, then by
Proposition~\ref{prop:correlations} in the one-prime case,
\[
  \md\bigl(A_E\cap A_E/a_1\cap\cdots\cap A_E/a_\ell\bigr)
  =d_0\Bigl(E\cap\bigcap_{i=1}^{\ell}(E-\vp(a_i))\Bigr).
\]
On the other hand, applying the first part of the theorem with
$T=\{0,\vp(a_1),\dots,\vp(a_\ell)\}$ gives
\[
  \mu_E\bigl(B\cap\sigma^{-\vp(a_1)}B\cap\cdots\cap\sigma^{-\vp(a_\ell)}B\bigr)
  =d_0\Bigl(E\cap\bigcap_{i=1}^{\ell}(E-\vp(a_i))\Bigr).
\]
This proves the last claim.
\end{proof}

The theorem above may be viewed as a correspondence principle: multiplicative
correlations on $\N$ are represented exactly as ergodic correlations on a
symbolic dynamical system. In contrast with affine approaches that mix addition
and multiplication on the ambient space, the valuation framework transfers
multiplicative structure to additive dynamics on exponent coordinates, where
correlations become explicitly computable.

Assume only that $d_0(E)$ exists, without assuming the existence of the
higher correlation densities. Compactness of the space of probability measures
on $X_E$ gives a subsequence $(N_j)_j$ for which
$\mu_{N_j}\to\mu$ weak-$*$. The same telescoping estimate used above shows that
$\mu$ is $\sigma$-invariant, and
\[
  \mu(B)=\lim_{j\to\infty}\mu_{N_j}(B)=d_0(E).
\]
For $a_1,\dots,a_\ell\in\N$, put
$T=\{0,\vp(a_1),\dots,\vp(a_\ell)\}$. Then
\[
\begin{aligned}
  \mu\Bigl(\bigcap_{t\in T}\sigma^{-t}B\Bigr)
  &=\lim_{j\to\infty}\frac1{N_j}
    \sum_{m=0}^{N_j-1}\prod_{t\in T}\one_E(m+t)\\
  &\le \limsup_{N\to\infty}\frac1N
    \sum_{m=0}^{N-1}\prod_{t\in T}\one_E(m+t)\\
  &=\umd_{(F_n)_n}
    \bigl(A_E\cap A_E/a_1\cap\cdots\cap A_E/a_\ell\bigr),
\end{aligned}
\]
where the last equality follows from the exact finite-box identity in
Theorem~\ref{thm:one-prime}(i), applied to
$E\cap\bigcap_{i=1}^{\ell}(E-\vp(a_i))$ and with the index correspondence
$N=n+1$. Thus the usual subsequential
correspondence principle yields a one-sided upper-density estimate; see, for
example, \cite{EinsiedlerWard,BergelsonMoreira,BergelsonMoreiraAffine}. Under
the hypotheses of Theorem~\ref{thm:correspondence}, the full empirical sequence
converges, every relevant multiplicative correlation density exists, and the
inequality becomes an exact identity. This permits asymptotic properties of
both $A_E$ and $E$ to be read from dynamical properties of
$(X_E,\mu_E,\sigma)$.

The hypothesis on $E$ may seem strong, since it requires the existence of
densities for all finite intersections of translates of $E$. The next
proposition shows that this requirement is automatically satisfied whenever
$x_E$ is generic for an invariant measure.

\begin{proposition}\label{prop:generic-implies-correlations}
Let $E\subseteq\Nz$, and let $\mu$ be a $\sigma$-invariant Borel probability
measure on $X_E$. If $x_E$ is generic for $(X_E,\mu,\sigma)$, then for every
finite set $T\subseteq\Z$, the asymptotic density
\[
d_0\left(\bigcap_{t\in T}(E-t)\right)
\]
exists and equals
\[
\mu\left(\bigcap_{t\in T}\sigma^{-t}B\right),
\]
where $B:=\{x\in X_E:x(0)=1\}$.
\end{proposition}
\begin{proof}
Fix a finite set $T\subseteq\Z$ and let
\[
  C_T:=\bigcap_{t\in T}\sigma^{-t}B.
\]
This is a cylinder set, so $\one_{C_T}$ is continuous. By the convention
following~\eqref{eq:def-xE}, for every $m\ge0$,
\[
  \one_{C_T}(\sigma^m x_E)
  =\prod_{t\in T}\one_E(m+t).
\]
Consequently,
\[
\begin{aligned}
  \frac1N\sum_{m=0}^{N-1}\one_{C_T}(\sigma^m x_E)
  &=\frac1N\sum_{m=0}^{N-1}\prod_{t\in T}\one_E(m+t)\\
  &=\frac{\left|\bigcap_{t\in T}(E-t)\cap[0,N-1]\right|}{N}.
\end{aligned}
\]
Since $x_E$ is generic for $\mu$, the left-hand side converges to
$\int\one_{C_T}\,d\mu=\mu(C_T)$. Hence the density on the right exists and
\[
  d_0\Bigl(\bigcap_{t\in T}(E-t)\Bigr)
  =\mu\Bigl(\bigcap_{t\in T}\sigma^{-t}B\Bigr).
\]
\end{proof}

By construction, if $x_E$ is generic for an invariant measure $\mu$, then the
measure $\mu_E$ obtained above coincides with $\mu$. More generally, start with
a generic point $y\in\{0,1\}^{\Z}$ and define
$E:=\{n\ge0:y(n)=1\}$. The points $x_E$ and $y$ agree on all nonnegative
coordinates. Hence, for a cylinder function supported on $\{-M,\dots,M\}$, the two
forward orbit sequences agree from time $M$ onward; their averages therefore
have the same limit. Since cylinder functions are uniformly dense, $x_E$ is
generic for the same measure. In
particular, generic points for weakly mixing non-mixing symbolic systems (\cite{Chaconmap}) or for
mixing symbolic systems produce valuation sets with the corresponding
correlation behavior. If $E$ is periodic on $\Nz$, its periodic bi-infinite extension
appears as a limit of forward shifts of $x_E$; the initial boundary at $0$ has
zero frequency. Thus $\mu_E$ is supported on the corresponding finite
$\sigma$-cycle, and the associated measure-preserving system is finite cyclic.

For every finite set $S\subseteq\Z$, write
\[
  c(S):=d_0\Bigl(\bigcap_{s\in S}(E-s)\Bigr).
\]

\begin{proposition}\label{prop:ergodic-characterization}
Let $E\subseteq\Nz$ and assume that, for every finite set $T\subseteq\Z$, the
asymptotic density $d_0\bigl(\bigcap_{t\in T}(E-t)\bigr)$ exists. Then the
system $(X_E,\mu_E,\sigma)$ constructed in
Theorem~\ref{thm:correspondence} is:
\begin{enumerate}[label=\textup{(\roman*)}]
\item ergodic if and only if, for all finite $S,S'\subseteq\Z$,
\[
  \lim_{N\to\infty}\frac1N\sum_{n=0}^{N-1}
  c\bigl(S\cup(S'+n)\bigr)=c(S)c(S');
\]
\item weakly mixing if and only if, for all finite $S,S'\subseteq\Z$,
\[
  \lim_{N\to\infty}\frac1N\sum_{n=0}^{N-1}
  \left|c\bigl(S\cup(S'+n)\bigr)-c(S)c(S')\right|=0;
\]
\item mixing if and only if, for all finite $S,S'\subseteq\Z$,
\[
  \lim_{n\to\infty}c\bigl(S\cup(S'+n)\bigr)=c(S)c(S').
\]
\end{enumerate}
\end{proposition}
\begin{proof}
Let $B =\{x: x(0) = 1\}$. By Theorem \ref{thm:correspondence}, for finite $T \subseteq \Z$, we have $\mu_E\left(\bigcap_{t\in T} \sigma^{-t}B \right) = c(T)$, and for $n$ large enough $S$ and $S' + n$ are disjoint, then 
$$\mu_E\left( \bigcap_{s\in S} \sigma^{- s} B \cap \sigma^{-n}\bigcap_{s' \in S'} \sigma^{-s'}B\right) = c(S\cup (S' + n)).$$
By inclusion-exclusion, it is enough to test ergodicity, weak mixing and mixing conditions on finite intersections of the sets \(\sigma^{-s}B\), since these sets generate the cylinder algebra of \(X_E\). Hence, the result holds.
\end{proof}

One may naturally wonder whether these properties are commonly satisfied.
The following random model gives a positive answer in a probabilistic sense.

\begin{proposition}\label{prop:generic-mixing}
Let $(Y_e)_{e\in\Z}$ be a stationary irreducible aperiodic Markov chain on
$\{0,1\}$ with transition matrix $P$ and stationary distribution $\pi$. Consider
the random set
\[
  E_\omega:=\{e\ge0:Y_e(\omega)=1\}.
\]
Then almost surely, for every finite $T\subseteq\Z$, the density
\[
  d_0\left(\bigcap_{t\in T}(E_\omega-t)\right)
\]
exists. Moreover, the inclusion $X_{E_\omega}\hookrightarrow\{0,1\}^{\Z}$ sends
$\mu_{E_\omega}$ to the stationary Markov measure $\mu_P$. In particular, the
associated measure-preserving system is measure-theoretically the stationary
Markov shift determined by $(\pi,P)$.
\end{proposition}
\begin{proof}
By inclusion-exclusion, it is enough to compare the two measures on cylinders
that require a prescribed finite set of coordinates to be equal to $1$. Fix a
nonempty finite set $T\subseteq\Z$ and put
\[
  M_T:=\max\{0,-\min T\}.
\]
For every $m\ge M_T$,
\[
  \prod_{t\in T}\one_{E_\omega}(m+t)
  =\prod_{t\in T}Y_{m+t}.
\]
Thus the corresponding averages differ by at most $M_T/N$. Since an
irreducible aperiodic stationary finite-state Markov chain is ergodic, the
ergodic theorem gives
\[
  \lim_{N\to\infty}\frac1N\sum_{m=0}^{N-1}
  \prod_{t\in T}\one_{E_\omega}(m+t)
  =\mathbb{E}\left[\prod_{t\in T}Y_t\right]
  =\mu_P\left(\bigcap_{t\in T}\sigma^{-t}[1]\right)
  \quad\text{almost surely}.
\]
(The assertion for $T=\varnothing$ is immediate.) The collection of finite
subsets of $\Z$ is countable, so we may take one full-measure event on which all
these limits hold simultaneously. On this event the hypothesis of
Theorem~\ref{thm:correspondence} is satisfied, and its formula shows that
$\mu_{E_\omega}$ and $\mu_P$ agree on every cylinder requiring specified
coordinates to be $1$. Inclusion--exclusion gives agreement on all cylinders.
Hence the inclusion of $X_{E_\omega}$ into the full shift sends
$\mu_{E_\omega}$ to $\mu_P$.
\end{proof}

For a finite-state irreducible aperiodic chain,
$P^n(i,j)\to\pi_j$ as $n\to\infty$. Joint probabilities of ordered coordinate
blocks can be written, by the Markov property, using transition matrices whose
powers are the gaps between successive blocks. Letting all those gaps tend to
infinity and applying the displayed convergence successively factors the joint
probability into the product of its marginals. Thus the stationary Markov shift
is mixing of all orders. Hence, almost surely, the system associated with
$E_\omega$ has this property.

\begin{corollary}\label{cor:mixing}
Let $E\subset\Nz$ satisfy the hypotheses of Theorem~\ref{thm:correspondence}, and
let $(X_E,\mu_E,\sigma)$ be the corresponding measure-preserving system. If the
system is mixing, then
\[
  \lim_{n\to\infty}\md\bigl(A_E\cap A_E/p^{\,n}\bigr)
  =\md(A_E)^2
  =d_0(E)^2
  =\lim_{n\to\infty}d_0\bigl(E\cap(E-n)\bigr).
\]
If the system is mixing of all orders, then for every $k\in\N$,
\[
  \md\bigl(A_E\cap A_E/p^{n_1}\cap\cdots\cap A_E/p^{n_k}\bigr)
  \longrightarrow \md(A_E)^{k+1},
\]
and
\[
  d_0\bigl(E\cap(E-n_1)\cap\cdots\cap(E-n_k)\bigr)
  \longrightarrow d_0(E)^{k+1},
\]
as $0<n_1<\cdots<n_k$ and
\[
  \min\{n_1,n_2-n_1,\ldots,n_k-n_{k-1}\}\longrightarrow\infty.
\]
\end{corollary}
\begin{proof}
By Theorem~\ref{thm:correspondence}, with $a=p^n$,
\[
  \md(A_E\cap A_E/p^n)=\mu_E(B\cap\sigma^{-n}B),
  \qquad \md(A_E)=\mu_E(B).
\]
Mixing gives
$\mu_E(B\cap\sigma^{-n}B)\to\mu_E(B)^2$, and the correspondence identity also
identifies the same quantity with $d_0(E\cap(E-n))$. This proves the first
assertion. For the higher-order statement, apply the corresponding
higher-order mixing limit to
\[
  B,\ \sigma^{-n_1}B,\ \ldots,\ \sigma^{-n_k}B
\]
and use Theorem~\ref{thm:correspondence} once more.
\end{proof}

\subsection{Finite-prime case}

We now combine the one-prime symbolic systems into a finite-prime dynamical
model.

\begin{theorem}\label{thm:correspondence-finite}
Let $p_1,\dots,p_k$ be distinct primes, and for each $i=1,\dots,k$ let
$E_i\subseteq\Nz$. We take multiplicative density with respect to a prime-box
F{\o}lner sequence associated to an enumeration of the primes whose first $k$
elements are $p_1,\dots,p_k$. Assume that for each $i$ and every finite set
$T\subseteq\Z$, the density of $\bigcap_{t\in T}(E_i-t)$ exists. Let
$(X_i,\mu_i,\sigma_i,B_i)$ be the one-prime symbolic system associated with $E_i$
by Theorem~\ref{thm:correspondence}. Set
\[
  X:=X_1\times\cdots\times X_k,\qquad
  \mu:=\mu_1\otimes\cdots\otimes\mu_k,\qquad
  B:=B_1\times\cdots\times B_k.
\]
For each $a\in\N$, define
\[
  T_a:=\sigma_1^{\,\vp[p_1](a)}\times\cdots\times\sigma_k^{\,\vp[p_k](a)}.
\]
Then $(T_a)_{a\in\N}$ is a measure-preserving semigroup action of $(\N,\cdot)$ on
$(X,\mu)$ and, if
\[
  A:=\{n\in\N:\vp[p_i](n)\in E_i\text{ for }i=1,\dots,k\},
\]
then for every $\ell\in\N$ and every $a_1,\dots,a_\ell\in\N$,
\[
  \md\bigl(A\cap A/a_1\cap\cdots\cap A/a_\ell\bigr)
  =\mu\bigl(B\cap T_{a_1}^{-1}B\cap\cdots\cap T_{a_\ell}^{-1}B\bigr).
\]
\end{theorem}

\begin{proof}
Each $T_a$ is measure-preserving because it is a product of powers of
measure-preserving transformations.
Also, for $a,b\in\N$,
\[
  T_{ab}=\sigma_1^{\,\vp[p_1](ab)}\times\cdots\times\sigma_k^{\,\vp[p_k](ab)}
  =\sigma_1^{\,\vp[p_1](a)+\vp[p_1](b)}\times\cdots\times\sigma_k^{\,\vp[p_k](a)+\vp[p_k](b)}
  =T_a\circ T_b,
\]
so $(T_a)_{a\in\N}$ defines a semigroup action of $(\N,\cdot)$.

Now fix $a_1,\dots,a_\ell\in\N$. For each $1\le i\le k$, set
$t_{j,i}:=\vp[p_i](a_j)$ for $1\le j\le\ell$ and $t_{0,i}:=0$. By
Proposition~\ref{prop:correlations},
\[
  \md\bigl(A\cap A/a_1\cap\cdots\cap A/a_\ell\bigr)
  =\prod_{i=1}^{k}d_0\Bigl(\bigcap_{j=0}^{\ell}(E_i-t_{j,i})\Bigr).
\]
On the other hand, by Theorem~\ref{thm:correspondence},
\[
  \mu_i\bigl(B_i\cap\sigma_i^{-t_{1,i}}B_i\cap\cdots\cap\sigma_i^{-t_{\ell,i}}B_i\bigr)
  =d_0\Bigl(\bigcap_{j=0}^{\ell}(E_i-t_{j,i})\Bigr).
\]
Since $B=B_1\times\cdots\times B_k$ and $\mu=\mu_1\otimes\cdots\otimes\mu_k$, the
left-hand side factors over the coordinates. Comparing this with the
multiplicative density formula above proves the claim.
\end{proof}


\section*{Acknowledgements}
 A.~Y. was supported in part by the Natural Sciences and
Engineering Research Council of Canada, NSERC (GR030571 and GR030540).


\end{document}